\documentclass[11pt]{article}
\usepackage{amssymb}
\newtheorem{precor}{{\bf Corollary}}

\newtheorem{precon}{{\bf Conjecture}}

\newtheorem{prealphcon}{{\bf Conjecture}}

\newtheorem{predefin}{{\bf Definition}}

\newtheorem{preexm}{{\bf Example}}

\newtheorem{preappl}{{\bf Application}}

\newtheorem{prelem}{{\bf Lemma}}

\newenvironment{lem}{\begin{prelem}{\hspace{-0.5
               em}{\bf.\ }}}{\end{prelem}}
\newtheorem{preproof}{{\bf Proof.\ }}

\newenvironment{proof}[1]{\begin{preproof}{\rm
               #1}\hfill{$\blacksquare$}}{\end{preproof}}
\newtheorem{pretheorem}{{\bf Theorem}}

\newenvironment{theorem}{\begin{pretheorem}{\hspace{-0.5
               em}{\bf.\ }}}{\end{pretheorem}}
\newtheorem{prealphtheorem}{{\bf Theorem}}

\newtheorem{prealphlem}{{\bf Lemma}}

\newtheorem{prepro}{{\bf Proposition}}

\newtheorem{preprb}{{\bf Problem}}

\newtheorem{prerem}{{\bf Remark}}

\newtheorem{preapp}{{\bf Application}}

\newtheorem{prequ}{{\bf Question}}

\def\conct[#1,#2]{\mbox {${#1} \leftrightarrow {#2}$}}
\def\dconct[#1,#2]{\mbox {${#1} \rightarrow {#2}$}}
\def\deg[#1,#2]{\mbox {$d_{_{#1}}(#2)$}}
\def\mindeg[#1]{\mbox {$\delta_{_{#1}}$}}
\def\maxdeg[#1]{\mbox {$\Delta_{_{#1}}$}}
\def\outdeg[#1,#2]{\mbox {$d_{_{#1}}^{^+}(#2)$}}
\def\minoutdeg[#1]{\mbox {$\delta_{_{#1}}^{^+}$}}
\def\maxoutdeg[#1]{\mbox {$\Delta_{_{#1}}^{^+}$}}
\def\indeg[#1,#2]{\mbox {$d_{_{#1}}^{^-}(#2)$}}
\def\minindeg[#1]{\mbox {$\delta_{_{#1}}^{^-}$}}
\def\maxindeg[#1]{\mbox {$\Delta_{_{#1}}^{^-}$}}

\def\dre[#1,#2,#3]{\mbox {${\cal E}^{^{#3}}(#1,#2)$}}
\def\var[#1,#2]{\mbox {${\rm Var}_{_{#1}}(#2)$}}
\def\ls[#1]{\mbox {$\xi^{^{#1}}$}}
\def\hom[#1,#2]{\mbox {${\rm Hom}({#1},{#2})$}}
\def\onvhom[#1,#2]{\mbox {${\rm Hom^{v}}(#1,#2)$}}
\def\onehom[#1,#2]{\mbox {${\rm Hom^{e}}(#1,#2)$}}
\def\core[#1]{\mbox {$#1^{^{\bullet}}$}}
\def\cay[#1,#2]{\mbox {${\rm Cay}({#1},{#2})$}}
\def\sch[#1,#2,#3]{\mbox {${\rm Sch}({#1},{#2},{#3})$}}
\def\cays[#1,#2]{\mbox {${\rm Cay_{s}}({#1},{#2})$}}
\def\dirc[#1]{\mbox {$\stackrel{\rightarrow}{C}_{_{#1}}$}}
\def\cycl[#1]{\mbox {${\bf Z}_{_{#1}}$}}

\begin{document}

\begin{center}
{\Large \bf On The b-Chromatic Number of Regular Graphs Without $4$-Cycle}\\
\vspace{0.3 cm}
{\bf Saeed Shaebani}\\
{\it Department of Mathematical Sciences}\\
{\it Institute for Advanced Studies in Basic Sciences {\rm (}IASBS{\rm )}}\\
{\it P.O. Box {\rm 45195-1159}, Zanjan, Iran}\\
{\tt s\_shaebani@iasbs.ac.ir}\\ \ \\
\end{center}
\begin{abstract}
\noindent The b-chromatic number of a graph $G$, denoted by
$\varphi(G)$, is the largest integer $k$ that $G$ admits a proper
$k$-coloring such that each color class has a vertex that is
adjacent to at least one vertex in each of the other color
classes. We prove that for each $d$-regular graph $G$ which
contains no $4$-cycle,
$\varphi(G)\geq\lfloor\frac{d+3}{2}\rfloor$ and if $G$ has a
triangle, then $\varphi(G)\geq\lfloor\frac{d+4}{2}\rfloor$. Also,
if $G$ is a $d$-regular graph which contains no $4$-cycle and
$diam(G)\geq6$, then $\varphi(G)=d+1$. Finally, we show that for
any $d$-regular graph $G$ which does not contain $4$-cycle and
$\kappa(G)\leq\frac{d+1}{2}$, $\varphi(G)=d+1$.
\\

\end{abstract}
\section{Introduction}
All graphs considered in this paper are finite and simple
(undirected, loopless and without multiple edges). Let $G=(V,E)$
be a graph. A coloring (proper coloring) of $G$ is called a
b-coloring of $G$ if each color class contains a vertex that is
adjacent to at least one vertex in each of the other color
classes. The b-chromatic number of $G$, denoted by $\varphi(G)$,
is the largest integer $k$ such that $G$ admits a b-coloring by
$k$ colors. The concept of b-coloring of graphs introduced by
Irving and Manlove in 1999 in \cite{irv} and has received
attention recently, for example, see $[1-26]$.

Let $G$ be a graph which is colored. Suppose that $v$ is a vertex
of $G$ whose color is $c$. We say that $v$ is a color-dominating
vertex or color $c$ realizes on $v$ if $v$ is adjacent to at
least one vertex in each of the other color classes.

It is obvious that for each graph $G$ with maximum degree
$\Delta(G)$, $\varphi(G)\leq\Delta(G)+1$. Kratochvil, Tuza and
Voigt in \cite{kratv} proved that every $d$-regular graph with at
least $d^{4}$ vertices satisfies $\varphi(G)=d+1$. Cabello and
Jakovac in \cite{ca.ja} reduced $d^{4}$ to $2d^{3}-d^{2}+d$.
These bounds show that for each natural number $d$, there are
only finite $d$-regular graphs such that $\varphi(G)\neq d+1$. El
Sahili and Kouider in \cite{sa.ko} asked whether it is true that
every $d$-regular graph of girth at least 5 satisfies
$\varphi(G)=d+1$. Blidia, Maffray and Zemir in \cite{b.m.z}
proved that b-chromatic number of the Petersen graph is 3 and
then conjectured that the Petersen graph is the only exception.
They proved this conjecture for $d\leq6$. Kouider in \cite{kou1}
proved that b-chromatic number of any $d$-regular graph of girth
at least 6 is $d+1$. El Sahili and Kouider in \cite{sa.ko} proved
that b-chromatic number of any $d$-regular graph of girth 5 that
contains no 6-cycle is $d+1$.  Cabello and Jakovac in
\cite{ca.ja} proved a celebrated theorem for the b-chromatic
number of regular graphs of girth 5 which guarantees that
b-chromatic number of $d-$regular graphs with girth at least 5 in
bounded below by a linear function of $d$. They proved that a
$d$-regular graph with girth at least 5 has b-chromatic number at
least $\lfloor\frac{d+1}{2}\rfloor$. Also, they proved that for
except small values of $d$, every connected $d$-regular graph that
contains no $4$-cycle and its diameter is at least $d$, has
b-chromatic number $d+1$.

In this paper, we discuss about b-chromatic number of $d$-regular
graphs with no $4$-cycles. First, in Section 2 we prove that
$\varphi(G)\geq\lfloor\frac{d+3}{2}\rfloor$ and if $G$ has a
triangle, then $\varphi(G)\geq\lfloor\frac{d+4}{2}\rfloor$. These
lower bounds are sharp for the Petersen graph. Also, in Section 3
we prove that if $diam(G)\geq6$, then $\varphi(G)=d+1$. Finally,
in Section 3 we show that if $\kappa(G)\leq\frac{d+1}{2}$, then
$\varphi(G)=d+1$. This upper bound for vertex connectivity is
sharp for the Petersen graph. Through this paper, for each
natural number $n$, define $[n]:=\{i|i\in \mathbb{N},1\leq i\leq
n \}$.

\section{A Lower Bound}

There are many $d$-regular graphs that contain $4$-cycles and
their b-chromatic number is not $d+1$. The most famous graphs
with this property are complete bipartite graphs $K_{d,d}$ whose
b-chromatic number equals 2, independent from $d$. Cabello and
Jakovac in \cite{ca.ja} proved that the b-chromatic number of
$d-$regular graphs with girth at least 5 in bounded below by a
linear function of $d$. They proved that a $d$-regular graph with
girth at least 5 has b-chromatic number at least
$\lfloor\frac{d+1}{2}\rfloor$. In this regard, they proved the
following technical lemma.

\begin{lem}\label{mainlemma}
{Let $H$ be a bipartite graph with partitions $U$ and $V$ such
that $|U|=|V|$. Let $u^{*}\in U$ and $v^{*}\in V$. If for each
vertex $x\in V(H)\setminus\{u^{*},v^{*}\}$,
$deg_{H}(x)\geq\frac{|V|}{2}$, $deg_{H}(u^{*})>0$ and
$deg_{H}(v^{*})>0$, then $H$ has a perfect matching. }
\end{lem}

With a slice refinement, we obtain the following lower bounds.

\begin{theorem}{ \label{mainthm} Let $G$ be a $d$-regular graph that contains no $4$-cycle.
Then $\varphi(G)\geq\lfloor\frac{d+3}{2}\rfloor$. Besides, If $G$ has a triangle,
then $\varphi(G)\geq\lfloor\frac{d+4}{2}\rfloor$. These lower
bounds are sharp for the Petersen graph. }
\end{theorem}

\begin{proof}
{ There is nothing to prove when $d\in \{0,1,2\}$. So, let us
suppose that $d\geq3$. Let $v$ be an arbitrary vertex of $G$.
Since $G$ contains no $4$-cycle, the maximum degree of the
induced subgraph of $G$ on $N_{G}(v)$ is at most one. If $d$ is
odd, there exists some vertex of $N_{G}(v)$ that is not adjacent
to any vertex of $N_{G}(v)$. If $d$ is even, let
$v_{1},v_{2},\ldots,v_{d}$ be an arbitrary ordering of $N_{G}(v)$
and if $d$ is odd, let $v_{1},v_{2},\ldots,v_{d}$ be an arbitrary
ordering of $N_{G}(v)$ such that $v_{\frac{d+1}{2}}$ does not
have any neighbors in $N_{G}(v)$.

We color all vertices of
$N_{G}(v)\bigcup(\bigcup_{j=1}^{\lfloor\frac{d+1}{2}\rfloor}N_{G}(v_{i}))$
by the colors of the set $[d+1]$ such that each vertex of
$\{v\}\bigcup\{v_{i}|1\leq i\leq\lfloor\frac{d+1}{2}\rfloor\}$
sees all colors of $[d+1]$ on its closed neighborhood. First, we
color the vertex $v$ by color $d+1$ and for each $1\leq i\leq d$,
we color $v_{i}$ by color $i$. For each $1\leq i\leq d$, let
$V_{i}:=N_{G}(v_{i})\setminus (\{v\}\bigcup N_{G}(v))$,
$C_{i}:=[d]\setminus ($the set of colors that are appeared on
$\{v_{i}\}\bigcup(N_{G}(v_{i})\bigcap N_{G}(v))$ and
$S_{i}:=\{v\}\bigcup N_{G}(v)\bigcup(\bigcup_{j=1}^{i}V_{j})$.
Obviously, $|V_{i}|=|C_{i}|=d-1$ or $|V_{i}|=|C_{i}|=d-2$. Also,
$|V_{i}|=|C_{i}|=d-2$ if and only if $|N_{G}(v_{i})\bigcap
N_{G}(v)|=1$. Also, since $G$ contains no $4$-cycle, for each
$1\leq i,j\leq d$, if $i\neq j$, then $V_{i}\bigcap V_{j}=\phi$.
Now, we follow $\lfloor\frac{d+1}{2}\rfloor$ steps inductively.
For each $1\leq i\leq\lfloor\frac{d+1}{2}\rfloor$, at $i$-th
step, we only color all vertices of $V_{i}$ by all colors of
$C_{i}$ injectively. Suppose by induction that $1\leq
i\leq\lfloor\frac{d+1}{2}\rfloor$ and for each $1\leq k\leq i-1$,
at $k$-th step, we have only colored all vertices of $V_{k}$ by
all colors of $C_{k}$ injectively is such a way that the
resulting partial coloring on $S_{k}$ is a proper partial
coloring. Now, at $i$-th step, we want to color only all vertices
of $V_{i}$ by all colors of $C_{i}$ injectively is such a way
that the resulting partial coloring on $S_{i}$ be a proper
partial coloring. We consider a bipartite graph $H_{i}$ with one
partition $V_{i}$ and the other partition $C_{i}$ that a vertex
$x\in V_{i}$ is adjacent to a color $c\in C_{i}$ in the graph
$H_{i}$ if and only if (in the graph $G$) $x$ does not have any
neighbors in $S_{i-1}$ already colored by $c$. Such a coloring of
all vertices of $V_{i}$ by all colors of $C_{i}$ (as mentioned)
exists if and only if $H_{i}$ has a perfect matching.

Let $x$ be an arbitrary element of $V_{i}$. The set of neighbors
of $x$ in the graph $G$ that were already colored, is a subset of
$\{v_{i}\}\bigcup(\bigcup_{j=1}^{i-1}V_{j})$. Since $G$ contains
no $4$-cycle, for each $1\leq j\leq i-1$, $x$ has at most one
neighbor in $V_{j}$. Also, the color of the vertex $v_{i}$ does
not belong to $C_{i}$. Therefore, $deg_{H_{i}}(x)\geq
|C_{i}|-(i-1)$. Also, since for each $1\leq j\leq i-1$, $v_{j}$
sees all colors of $[d+1]$ on its closed neighborhood, each color
of $[d+1]$ appears at most once on $V_{j}$. Therefore, for each
$c\in C_{i}$, $deg_{H_{i}}(c)\geq|V_{i}|-(i-1)$. Hence, for each
$a\in V(H_{i})$, $deg_{H_{i}}(a)\geq|V_{i}|-(i-1)$. Since
$|V_{i}|\geq d-2$, if $1\leq i\leq\frac{d}{2}$, then
$|V_{i}|-(i-1)\geq\frac{|V_{i}|}{2}$. Therefore,
$deg_{H_{i}}(a)\geq\frac{|V_{i}|}{2}$. If
$\frac{d}{2}<i\leq\lfloor\frac{d+1}{2}\rfloor$, then $d$ is odd
and $i=\frac{d+1}{2}$. Since $v_{\frac{d+1}{2}}$ does not have any
neighbors in $N_{G}(v)$, $|V_{\frac{d+1}{2}}|=d-1$ and therefore,
$deg_{H_{\frac{d+1}{2}}}(a)\geq|V_{\frac{d+1}{2}}|-(\frac{d+1}{2}-1)=\frac{d-1}{2}=\frac{|V_{\frac{d+1}{2}}|}{2}$.

So, Lemma \ref{mainlemma} implies that $H_{i}$ has a perfect
matching and we are done. We conclude that there exists a partial
coloring on
$N_{G}(v)\bigcup(\bigcup_{j=1}^{\lfloor\frac{d+1}{2}\rfloor}N_{G}(v_{i}))$
by all colors of the set $[d+1]$ such that each vertex of
$\{v\}\bigcup\{v_{i}|1\leq i\leq\lfloor\frac{d+1}{2}\rfloor\}$
sees all colors of $[d+1]$ on its closed neighborhood. This
coloring extends to a coloring of $G$ greedily. In the current
coloring of $G$, all colors of the set
$[\lfloor\frac{d+1}{2}\rfloor]\bigcup\{d+1\}$ are realized. If
there exists a color $c\in [d+1]$ that does not realize in this
coloring, then for each vertex $a$ in $V(G)$, which is colored by
$c$, there exists a color (like $e$) in $[d+1]\setminus \{c\}$
such that $a$ does not have any neighbors in $G$ colored by $e$.
We exchange the color of the vertex $a$ by color $e$ (we recolor
the vertex $a$ by color $e$). The resulting coloring is a
coloring of $G$ such that the color of each vertex whose color was
different from $c$ in the previous coloring, has not been
changed. Also, the color $c$ does not appear in the current
coloring and therefore, the number of colors reduces. Each vertex
that was a color-dominating vertex in the previous coloring, is
again a color-dominating vertex in the current coloring.
Repeating this procedure, there exists a b-coloring of $G$ with
at least $\lfloor\frac{d+3}{2}\rfloor$ colors and therefore,
$\varphi(G)\geq\lfloor\frac{d+3}{2}\rfloor$.

Now, assume that $G$ has a triangle. If $d$ is odd, then
$\lfloor\frac{d+3}{2}\rfloor=\lfloor\frac{d+4}{2}\rfloor$.
Consequently, we suppose that $d$ is even. Choose a vertex $v$
that is in a triangle. Let $v_{1},v_{2},\ldots,v_{d}$ be an
arbitrary ordering of $N_{G}(v)$ such that $\{v_{1},
v_{\frac{d+2}{2}}\}\in E(G)$. There are not any edges between
$V_{1}$ and $V_{\frac{d+2}{2}}$, otherwise $v_{1}$ and
$v_{\frac{d+2}{2}}$ will be contained in a $C_{4}$. Therefore,
degree of each vertex of $H_{\frac{d+2}{2}}$ is at least
$(d-2)-(\frac{d+2}{2}-2)={|V_{d+2 \over 2}|\over 2}$, as desired.}
\end{proof}

For each edge $\{u,v\}\in E(G)$, let $max_{\{u,v\}}$ be the
maximum number of $5$-cycles such that the intersection of any two
different elements of them is the edge $\{u,v\}$. Also, for each
path $P$ of length two where $V(P)=\{u,v,w\}$, set $max_P$, the
maximum number of $5$-cycles such that the intersection of any two
different elements of them is the path $P$. The following theorems
can be proved similarly. We omit their proofs for the sake of
brevity.

\begin{theorem}{ Let $G$ be a $d$-regular graph that contains no $4$-cycle. If
there exists a vertex $v\in V(G)$ such that for each $\{u,v\}\in
E(G)$, the number of $5$-cycles that contain the edge $\{u,v\}$
is less than or equal to $\frac{d-2}{2}$, then $\varphi(G)=d+1$.
Besides, if girth of $G$ is $5$, $\frac{d-2}{2}$ can be replaced
by $\frac{d-1}{2}$. }
\end{theorem}

\begin{theorem}{ Let $G$ be a $d$-regular graph that contains no $4$-cycle. If
there exists a vertex $v\in V(G)$ such that for each $\{u,v\}\in
E(G)$, $max_{\{u,v\}}\leq \frac{d-2}{2}$ and for each path $P$ of
length two, $max_P\leq \frac{d-2}{2}$, then $\varphi(G)=d+1$.
Besides, if girth of $G$ is 5, $\frac{d-2}{2}$ can be replaced by
$\frac{d-1}{2}$.
 }
\end{theorem}

\section{Diameter}

Cabello and Jakovac in \cite{ca.ja} proved that for $d\geq 10$,
every connected $d$-regular graph that contains no $4$-cycle and
its diameter is at least $d$, has b-chromatic number $d+1$. In
this section, we determine the b-chromatic number of $d$-regular
graphs that contain no $4$-cycles and have diameter at least $6$.

\begin{theorem}{ Let $G$ be a $d$-regular graph that contains no $4$-cycle. If $diam(G)\geq6$, then $\varphi(G)=d+1$. }
\end{theorem}
\begin{proof}
{There is nothing to prove when $d\in \{0,1,2\}$. So, suppose
that $d\geq3$. Since $diam(G)\geq6$, there are two vertices $v$
and $w$ of distance at least 6. According to the proof of Theorem
\ref{mainthm}, there exists a partial coloring (using all colors
of the set $[d+1]$) on a subset $V$ of vertices that are at
distance at most two from $v$ such that all colors of the set
$[\lfloor\frac{d+3}{2}\rfloor]$ are realized. Also, there exists
a partial coloring (using all colors of the set $[d+1]$) on a
subset $W$ of vertices that are at distance at most two from $w$
such that all colors of the set $[d+1]\setminus
[\lfloor\frac{d+3}{2}\rfloor]$ do realize. Since $v$ and $w$ have
distance at least 6, there are not any edges between $V$ and $W$.
Therefore, these two partial colorings do not conflict and
therefore, they form a partial coloring on $V\bigcup W$ that all
colors of the set $[d+1]$ are realized. This partial coloring
extends to a coloring of $G$ greedily and therefore,
$\varphi(G)=d+1$. }
\end{proof}

\section{Vertex Connectivity}

Vertex connectivity of a graph $G$, denoted by $\kappa(G)$, is
the minimum cardinality of a subset $U$ of $V(G)$ such that
$G\setminus U$ is either disconnected or a graph with only one
vertex. It is well-known that for each graph $G$, $\kappa(G)\leq
\delta(G)$ where $\delta(G)$ denotes the minimum degree of $G$.
In this section, we show that if $G$ is a $d$-regular graph that
contains no $4$-cycle and $\kappa(G)\leq \frac{d+1}{2}$, then
$\varphi(G)=d+1$. This upper bound for vertex connectivity is
sharp in the sense that the vertex connectivity of the Petersen
graph is $\frac{d+1}{2}+1$ although its b-chromatic number is not
$d+1$.

\begin{theorem}{ Let $G$ be a $d$-regular graph that contains no $4$-cycle. If $\kappa(G)\leq \frac{d+1}{2}$,
then $\varphi(G)=d+1$. This upper
bound for vertex connectivity is sharp for the Petersen graph. }
\end{theorem}
\begin{proof}
{There is nothing to prove when $d\in \{0,1,2\}$. Also, Jakovac
and Klavzar in \cite{ja.kl} showed that the only cubic graph that
contains no $4$-cycle and its b-chromatic number is not $4$ is the
Petersen graph. Since the vertex connectivity of the Petersen
graph is 3, the proof is complete for $d=3$. So, suppose that
$d\geq4$. Let $U$ be a set of minimum cardinality such that
$G\setminus U$ is either disconnected or a graph with only one
vertex. Obviously, $|U|=\kappa(G)$. Since $G$ is not a complete
graph, $G\setminus U$ is disconnected. Let $G_{1}$ and $G_{2}$ be
two distinct connected components of $G\setminus U$. Since
$\kappa(G)\leq \frac{d+1}{2}<d$, $|G_{i}|\geq2\ (i=1,2)$. For each
$i\in \{1,2\}$, we prove that there exists some $a_{i}\in G_{i}$
such that there are no edges between $a_{i}$ and $U$. In this
regard, we consider two cases, $d\in \{4,6,7,8,9,\ldots\}$ or
$d=5$.

Case $1$) The case $d\geq4$ and $d\neq5$: There is nothing to
prove when $U=\phi$ (or equivalently $G$ is disconnected). So, we
consider the case $U\neq\phi$. There exists $b_{i}\in G_{i}$ such
that $|N_{G}(b_{i})\bigcap U|\leq \frac{\kappa(G)+1}{2}$,
otherwise since $|G_{i}|\geq2$ and $G$ contains no $4$-cycle,
inclusion-exclusion principle implies that
$|U|>2(\frac{\kappa(G)+1}{2})-1=\kappa(G)$, a contradiction. Let
$U:=\{u_{j}|1\leq j\leq \kappa(G)\}$ and for each $1\leq j\leq
\kappa(G)$, set $A_{j}^{i}:=\{x|x\in G_{i}, \{x,u_{j}\}\in
E(G)\}$. For each $1\leq j\leq\kappa(G)$, $b_{i}$ has at most one
neighbor in $A_{j}^{i}$, otherwise $b_{i}$ and $u_{j}$ will be
contained in a $C_{4}$. Therefore,
$|N_{G}(b_{i})\bigcap(U\bigcup(\bigcup_{j=1}^{\kappa(G)}A_{j}^{i}))|\leq\frac{\kappa(G)+1}{2}+\kappa(G)=\frac{3\kappa(G)+1}{2}$.
If $d=4$, $\kappa(G)\leq\frac{d+1}{2}=\frac{5}{2}$ and therefore,
$\kappa(G)\leq 2$. Hence,
$|N_{G}(b_{i})\bigcap(U\bigcup(\bigcup_{j=1}^{\kappa(G)}A_{j}^{i}))|\leq\frac{7}{2}<4$
and therefore, there exists some vertex in $N_{G}(b_{i})$ which is
not in $U\bigcup(\bigcup_{j=1}^{\kappa(G)}A_{j}^{i})$.
Accordingly there exists some $a_{i}\in G_{i}$ such that
$N_{G}(a_{i})\bigcap U=\phi$. If $k>5$, then
$\frac{3\kappa(G)+1}{2}\leq\frac{\frac{3}{2}(d+1)+1}{2}=d+\frac{5-d}{4}<d$,
hence, there exists some $a_{i}\in G_{i}$ such that
$N_{G}(a_{i})\bigcap U=\phi$.

Case $2$) The case $d=5$: If $U=\phi$, there is nothing to prove.
If $U\neq\phi$, similar to the previous case, there exists
$b_{i}\in G_{i}$ such that $|N_{G}(b_{i})\bigcap
U|\leq\frac{\kappa(G)+1}{2}$. Also, let $U:=\{u_{j}|1\leq
j\leq\kappa(G)\}$ and for each $1\leq j\leq\kappa(G)$, define
$A_{j}^{i}:=\{x|x\in G_{i}, \{x,u_{j}\}\in E(G)\}$. Since for
each $1\leq j\leq\kappa(G)$, $b_{i}$ has at most one neighbor in
$A_{j}^{i}$,
$|N_{G}(b_{i})\bigcap(U\bigcup(\bigcup_{j=1}^{\kappa(G)}A_{j}^{i}))|=|(N_{G}(b_{i})
\bigcap
U)\bigcup(N_{G}(b_{i})\bigcap(\bigcup_{j=1}^{\kappa(G)}A_{j}^{i}))|\leq|N_{G}(b_{i})
\bigcap
U|+|N_{G}(b_{i})\bigcap(\bigcup_{j=1}^{\kappa(G)}A_{j}^{i})|\leq$$\frac{\kappa(G)+1}{2}+\kappa(G)$.
Obviously if $\kappa(G)<3$ or $|N_{G}(b_{i})\bigcap U|<2$ or
$|N_{G}(b_{i})\bigcap(\bigcup_{j=1}^{\kappa(G)}A_{j}^{i})|<3$,
then
$|N_{G}(b_{i})\bigcap(U\bigcup(\bigcup_{j=1}^{\kappa(G)}A_{j}^{i}))|<5$
and we are done. If none of these three conditions hold, we
conclude that $\kappa(G)=3$, $|N_{G}(b_{i})\bigcap U|\geq2$, and
$|N_{G}(b_{i})\bigcap(\bigcup_{j=1}^{\kappa(G)}A_{j}^{i})|\geq3$.
Since
$|N_{G}(b_{i})\bigcap(\bigcup_{j=1}^{\kappa(G)}A_{j}^{i})|\geq3$,
$G_{i}$ has at least three elements $b_{i_{1}}$,$b_{i_{2}}$ and
$b_{i_{3}}$ different from $b_{i}$. Since $|U|=\kappa(G)=3$ and
${{\kappa(G) \choose 2}}={{3 \choose 2}}=3$ and $G$ contains no
$4$-cycle, there exists some $1\leq l\leq 3$ such that
$|N_{G}(b_{i_{l}})\bigcap U|\leq 1$. Since $b_{i_{l}}$ has at most
three neighbors in $\bigcup_{j=1}^{\kappa(G)}A_{j}^{i}$,
therefore,
$|N_{G}(b_{i_{l}})\bigcap(U\bigcup(\bigcup_{j=1}^{\kappa(G)}A_{j}^{i}))|\leq
4$. Hence, there exists some $a_{i}\in G_{i}$ such that
$N_{G}(a_{i})\bigcap U=\phi$.

We conclude that there exist some $a_{1}\in G_{1}$ and some
$a_{2}\in G_{2}$ such that $N_{G}(a_{1})\bigcap U=\phi$ and
$N_{G}(a_{2})\bigcap U=\phi$. For each $i\in \{1,2\}$, $a_{i}$ has
at most $\kappa(G)$ neighbors in
$\bigcup_{j=1}^{\kappa(G)}A_{j}^{i}$, therefore,
$|N_{G}(a_{i})\setminus(\bigcup_{j=1}^{\kappa(G)}A_{j}^{i})|\geq
d-\kappa(G)\geq d-\frac{d+1}{2}=\frac{d-1}{2}$.

For coloring $G$, we consider two cases, $d$ is even or odd.

The case $d$ is even) In this case,
$|N_{G}(a_{i})\setminus(\bigcup_{j=1}^{\kappa(G)}A_{j}^{i})|\geq\frac{d-1}{2}$,
so,
$|N_{G}(a_{i})\setminus(\bigcup_{j=1}^{\kappa(G)}A_{j}^{i})|\geq\frac{d}{2}$.
Let $x_{1},x_{2},\ldots,x_{\frac{d}{2}}$ be $\frac{d}{2}$
different elements of
$N_{G}(a_{1})\setminus(\bigcup_{j=1}^{\kappa(G)}A_{j}^{1})$. We
color $a_{1}$ by color 1 and for each $1\leq j\leq\frac{d}{2}$, we
color $x_{i}$ by color $i+1$. Then we color all elements of
$N_{G}(a_{1})\setminus\{x_{1},x_{2},\ldots,x_{\frac{d}{2}}\}$ by
all colors of the set $[d+1]\setminus[\frac{d}{2}+1]$
injectively. Therefore, $a_{1}$ is a color-dominating vertex with
color 1. For each $1\leq i\leq\frac{d}{2}$, let
$V_{i}:=N_{G}(x_{i})\setminus(\{a_{1}\}\bigcup N_{G}(a_{1}))$,
$C_{i}:=\{2,\ldots,d+1\}\setminus($the set of colors that were
already appeared on $\{x_{i}\}\bigcup(N_{G}(x_{i})\bigcap
N_{G}(a_{1}))$, and $S_{i}:=\{a_{1}\}\bigcup
N_{G}(a_{1})\bigcup(\bigcup_{j=1}^{i}V_{j})$. Obviously,
$V_{i}\subseteq G_{1}$. Besides, $|V_{i}|=|C_{i}|=d-1$ or
$|V_{i}|=|C_{i}|=d-2$. Also, $|V_{i}|=|C_{i}|=d-2$ if and only if
$|N_{G}(x_{i})\bigcap N_{G}(a_{1})|=1$. Also, for each $1\leq
i,j\leq\frac{d}{2}$, if $i\neq j$, then $V_{i}\bigcap V_{j}=\phi$.
Now, we follow $\frac{d}{2}$ steps inductively. For each $1\leq
i\leq\frac{d}{2}$, at $i$-th step, we only color all vertices of
$V_{i}$ by all colors of $C_{i}$ injectively to make $x_{i}$ a
color-dominating vertex. Suppose by induction that $1\leq
i\leq\frac{d}{2}$ and for each $1\leq k\leq i-1$, at $k$-th step,
we have only colored all vertices of $V_{k}$ by all colors of
$C_{k}$ injectively in such a way that the resulting partial
coloring on $S_{k}$ is a proper partial coloring. Now, at $i$-th
step, we want to color only all vertices of $V_{i}$ by all colors
of $C_{i}$ injectively in such a way that the resulting partial
coloring on $S_{i}$ be a proper partial coloring. We consider a
bipartite graph $H_{i}$ with one partition $V_{i}$ and the other
partition $C_{i}$ such that a vertex $v\in V_{i}$ is adjacent to
a color $c\in C_{i}$ in the graph $H_{i}$ if and only if (in the
graph $G$) $v$ does not have any neighbors in $S_{i}$ already
colored by $c$. Such a coloring of all vertices of $V_{i}$ by all
colors of $C_{i}$ (as mentioned) exists if and only if $H_{i}$ has
a perfect matching.

Let $v$ be an arbitrary element of $V_{i}$. The set of neighbors
of $v$ in the graph $G$ that were already colored, is a subset of
$\{x_{i}\}\bigcup(\bigcup_{j=1}^{i-1}V_{j})$. Since $G$ contains
no $4$-cycle, for each $1\leq j\leq i-1$, $v$ has at most one
neighbor in $V_{j}$. Also, the color of the vertex $x_{i}$ does
not belong to $C_{i}$. Therefore,
$deg_{H_{i}}(v)\geq|C_{i}|-(i-1)$. Also, since for each $1\leq
j\leq i-1$, $x_{j}$ is a color-dominating vertex, each color of
$[d+1]$ appears at most once on $V_{j}$. Therefore, for each $c\in
C_{i}$, $deg_{H_{i}}(c)\geq|V_{i}|-(i-1)$. Hence, for each $x\in
V(H_{i})$, $deg_{H_{i}}(x)\geq|V_{i}|-(i-1)$. Since $|V_{i}|\geq
d-2$ and $1\leq i\leq\frac{d}{2}$,
$|V_{i}|-(i-1)\geq\frac{|V_{i}|}{2}$ and therefore,
$deg_{H_{i}}(x)\geq\frac{|V_{i}|}{2}$. So, Lemma \ref{mainlemma}
implies that $H_{i}$ has a perfect matching and we are done.

We conclude that we have a partial coloring on $S_{\frac{d}{2}}$
such that all colors of $[\frac{d}{2}+1]$ are realized.

Since $G_{1}$ and $G_{2}$ are two distinct connected components of
$G\setminus U$, there are not any edges between $G_{1}$ and
$G_{2}$ in the graph $G$. Similarly, there exists a partial
coloring on a subset of $G_{2}$ such that all colors of
$[d+1]\setminus[\frac{d}{2}+1]$ do realize. Since there are not
any edges between $G_{1}$ and $G_{2}$ in the graph $G$, these two
different partial colorings, form a partial coloring in the graph
$G$ such that all colors of $[d+1]$ are realized. This partial
coloring extends to a coloring of $G$ greedily and therefore,
$\varphi(G)=d+1$.

The case $d$ is odd) In this case, the proof is similar to the
previous case. For each $i\in \{1,2\}$,
$|N_{G}(a_{i})\setminus(\bigcup_{j=1}^{\kappa(G)}A_{j}^{i})|\geq\frac{d-1}{2}$.
There exists a partial coloring on a subset of $G_{1}$ such that
all colors of $[\frac{d+1}{2}]$ are realized. Also, there exists a
partial coloring on a subset of $G_{2}$ such that all colors of
$[d+1]\setminus[\frac{d+1}{2}]$ are realized. Since there are not
any edges between $G_{1}$ and $G_{2}$ in the graph $G$, these two
different partial colorings, form a partial coloring in the graph
$G$ such that all colors of $[d+1]$ are realized. This partial
coloring extends to a coloring of $G$ greedily and therefore,
$\varphi(G)=d+1$. }
\end{proof}

The following theorem can be proved similarly. We omit its proof
for the sake of brevity.

\begin{theorem}{ Let $d\in\mathbb{N}\bigcup\{0\}$ and $G$ be a $d$-regular graph
that contains no $4$-cycle. If $\kappa(G)<\frac{2d-1}{3}$, then
$min\{2\lfloor\frac{d+4}{3}\rfloor,d+1\}\leq\varphi(G)\leq d+1$.
Besides, if there exists a set $U\subseteq V(G)$ such that
$|U|=\kappa(G)$ and $G\setminus U$ has at least three components,
then $\varphi(G)=d+1$. }
\end{theorem}

\ \\
{\bf Acknowledgement:} The author wishes to thank Professor
Hossein Hajiabolhassan for his useful comments.

\end{document}